\documentclass[reqno,11pt]{amsart}

\usepackage[all]{xy}
\usepackage{amssymb}
\usepackage{amsgen}
\usepackage{amsmath}
\usepackage{amsthm}
\usepackage{mathrsfs}
\usepackage{cite}
\usepackage{amsfonts}

\newcommand{\sour}{\mathop{\boldsymbol s}}
\newcommand{\dom}{\mathop{\boldsymbol d}}
\newcommand{\ran}{\mathop{\boldsymbol r}}

\newcommand{\inv}{^{-1}}

\usepackage{xcolor}

\newcommand{\skel}[1]{^{(#1)}}

\newtheorem{Thm}{Theorem}
\newtheorem{Prop}[Thm]{Proposition}

\newtheorem{Lemma}[Thm]{Lemma}
{\theoremstyle{definition}
}
{\theoremstyle{remark}
}

{\theoremstyle{remark}
}
{\theoremstyle{remark}
}

{\theoremstyle{remark}
}
{\theoremstyle{remark}
}
{\theoremstyle{remark}
}

\numberwithin{equation}{section}

\title[Chain conditions on \'etale groupoid algebras]{Chain conditions on \'etale groupoid algebras with applications to Leavitt path algebras and inverse semigroup algebras}

\author{Benjamin Steinberg}
\address{%
    Department of Mathematics\\
    City College of New York\\
    Convent Avenue at 138th Street\\
    New York, New York 10031\\
    USA}
\email{bsteinberg@ccny.cuny.edu}

\thanks{The author was supported by  NSA MSP \#H98230-16-1-0047.}
\date{September 8, 2017}


\begin{document}

\begin{abstract}
The author has previously associated to each commutative ring with unit $R$ and \'etale groupoid $\mathscr G$ with locally compact, Hausdorff and totally disconnected unit space an $R$-algebra $R\mathscr G$.  In this paper we characterize when $R\mathscr G$ is Noetherian and when it is Artinian.  As corollaries, we extend the characterization of Abrams, Aranda~Pino and Siles~Molina of finite dimensional and of Noetherian Leavitt path algebras over a field to arbitrary commutative coefficient rings and we recover the characterization of Okni\'nski of Noetherian  inverse semigroup algebras and of Zelmanov of Artinian inverse semigroup algebras.
\end{abstract}

\maketitle

\section{Introduction}
Groupoid $C^*$-algebras have played a crucial role in operator algebra theory since the seminal work of Renault~\cite{Renault}; see also~\cite{Paterson,Exel}.  Recently, the author introduced~\cite{mygroupoidalgebra,mygroupoidarxiv} a ring theoretic analogue for the class of ample groupoids~\cite{Renault,Paterson} over any base commutative ring with unit.  Ample groupoids are \'etale groupoids with a locally compact, Hausdorff and totally disconnected unit space where we recall that a topological groupoid is \'etale if its structure maps are local homeomorphisms.  Note that these algebras were independently introduced over the field of complex numbers slightly later in~\cite{operatorguys1}.

Algebras of ample groupoids, dubbed Steinberg algebras by Clark and Sims~\cite{GroupoidMorita}, include many important classes of rings including group algebras, commutative algebras over a field generated by idempotents, crossed products of the previous two sorts of rings, inverse semigroup algebras and Leavitt path algebras~\cite{Leavittbook}.  Over the past few years, there has been a plethora of papers on these algebras,  particularly in connection with Leavitt path algebras, cf.~\cite{operatorguys1,operatorsimple1,operatorguys2,reconstruct,GroupoidMorita,CarlsenSteinberg,Strongeffective,ClarkPardoSteinberg,CenterLeavittGroupoid,groupoidbundles,groupoidprimitive,Clarkdecomp,GonRoy2017a,GonRoy2017,Hazrat2017,AraSims2017,Purelysimple}.

In this paper we investigate the ascending and descending chain conditions on ample groupoid algebras.  A classical result of Connell~\cite{connell} says that a group ring $RG$ is Artinian if and only if $G$ is finite and $R$ is Artinian.  This was (essentially) extended to semigroups by Zelmanov~\cite{Zelmanovsgp}.  The finite dimensional Leavitt path algebras over a field were characterized by Abrams, Aranda~Pino and Siles~Molina~\cite{LeavittArtinian}.

The problem of characterizing Noetherian group rings is more complicated.  The largest class of groups known to have Noetherian group algebras over a Noetherian coefficient ring is the class of polycyclic-by-finite groups.  Thus the best we can hope to achieve for groupoid algebras is to classify the Noetherian property up to the case of group rings, which we achieve.  Special cases of our results include the characterization of Noetherian Leavitt path algebras over a field~\cite{LeavittNoeth} (which we now extend to any base ring) and Okni\'nski's characterization of Noetherian inverse semigroup algebras~\cite{Noetherianokninskiinverse}.

Groupoid algebras admit an involution and hence are isomorphic to their opposite algebras.  Thus there is no difference for them between the left and right Noetherian or Artinian properties and so we omit the adjectives ``left" and ``right" in what follows.
Our main result is the following theorem (see below for undefined terminology).

\begin{Thm}\label{t:main}
Let $\mathscr G$ be an ample groupoid and $R$ a commutative ring with unit.
\begin{enumerate}
\item   The groupoid algebra $R\mathscr G$ is Noetherian if and only if $\mathscr G$ has finitely many objects and $RG$ is Noetherian for each isotropy group $G$ of $\mathscr G$.
\item  The groupoid algebra $R\mathscr G$ is Artinian if and only if $\mathscr G$ is finite and $R$ is Artinian.
\item The groupoid algebra $R\mathscr G$ is semisimple if and only if $\mathscr G$ is finite and $R$ is a finite direct product of fields whose characteristics do not divide the order of any isotropy subgroup of $\mathscr G$.
\end{enumerate}
Moreover, in any of these cases $R\mathscr G$ is a finite direct product of matrix algebras over group algebras $RG$ of isotropy groups $G$ of $\mathscr G$.
\end{Thm}

Theorem~\ref{t:main} is in a sense a negative result since it indicates that \'etale groupoid algebras satisfy chain conditions only under very stringent hypotheses.  Nonetheless, we recover the known results for Leavitt path algebras and inverse semigroup algebras as special cases.

\section{\'Etale groupoids and their algebras}
In this paper, following the Bourbaki convention, a topological space will be called compact if it is Hausdorff and satisfies the property that every open cover has a finite subcover.

\subsection{\'Etale groupoids}
A \emph{topological groupoid} $\mathscr G$ is a groupoid (i.e., a small category each of whose morphisms is an isomorphism) whose object (or unit) space $\mathscr G\skel 0$ and arrow space $\mathscr G\skel 1$ are topological spaces and whose domain map $\dom$, range map $\ran$, multiplication map, inversion map and unit map $u\colon \mathscr G\skel 0\to \mathscr G\skel 1$ are all continuous.  Since $u$ is a homeomorphism with its image, we often identify elements of $\mathscr G\skel 0$ with the corresponding identity arrows and view $\mathscr G\skel 0$ as a subspace of $\mathscr G\skel 1$ with the subspace topology. 

A topological groupoid $\mathscr G$ is \emph{\'etale} if $\dom$ is a local homeomorphism.  This implies that $\ran$ and the multiplication map are local homeomorphisms  and that $\mathscr G\skel 0$ is open in $\mathscr G\skel 1$~\cite{resendeetale}.  Note that the fibers of $\dom$ and $\ran$ are discrete in the induced topology. 

An \'etale groupoid is said to be \emph{ample}\cite{Paterson} if $\mathscr G\skel 0$ is Hausdorff and has a basis of compact open sets.  In this case $\mathscr G\skel 1$ is locally Hausdorff but need not be Hausdorff.    Note that any discrete groupoid is ample. A discrete group is the same thing as an ample gropoid with a single object.

If $x\in \mathscr G\skel 0$, then the \emph{isotropy group} $G_x$ at $x$ is the group of all arrows $g\colon x\to x$.  The \emph{orbit} $\mathcal O_x$ of $x\in \mathcal G\skel 0$ is the set of objects $y$ such that there exists an arrow $g\colon x\to y$.  The isotropy groups of elements in the same orbit are isomorphic.

\subsection{Groupoid algebras}
Let $\mathscr G$ be an ample groupoid and $R$ a commutative ring with unit.  Define $R\mathscr G$ to be the $R$-submodule of $R^{\mathscr G\skel 1}$ spanned by the characteristic functions $\chi_U$ with $U\subseteq \mathscr G\skel 1$ compact open.  If $\mathscr G\skel 1$ is Hausdorff, then $R\mathscr G$ consists precisely of the compactly supported, locally constant functions $\mathscr G\skel 1\to R$. See~\cite{mygroupoidalgebra,mygroupoidarxiv,operatorguys1} for details.

The convolution product on $R\mathscr G$, defined by
\[f_1\ast f_2(g)=\sum_{\dom(h)=\dom(g)}  f_1(gh\inv)f_2(h), \] turns $R\mathscr G$ into an $R$-algebra.  Note that if $U,V\subseteq \mathscr G\skel 0$ are compact open, then $\chi_U\ast\chi_V=\chi_{U\cap V}$.
The ring $R\mathscr G$ is unital if and only if $\mathscr G\skel 0$ is compact. 

If $\mathscr G$ is discrete, then the elements of $R\mathscr G$ are just the finitely supported functions $\mathscr G\skel 1\to R$ and we can identify $R\mathscr G$ with the category algebra of $\mathscr G$ (in the sense of Mitchell~\cite{ringoids,Webb}).  That is we can identify the underlying $R$-module with the free $R$-module $R\mathscr G\skel 1$ on $\mathscr G\skel 1$ and equip it with the unique product extending the product on basis elements given by
\[g\cdot h = \begin{cases} gh, & \text{if}\ \dom(g)=\ran(h)\\ 0, & \text{else.}\end{cases}\]

\section{Proof of Theorem~\ref{t:main}}
We shall first prove Theorem~\ref{t:main} under the hypothesis that $\mathscr G\skel 0$ is finite.  Then we shall show that the Noetherian (and hence Artinian) condition implies that $\mathscr G\skel 0$ is finite.

\begin{Prop}\label{p:finitelymanyobs}
Suppose that $\mathscr G$ is an ample groupoid with $\mathscr G\skel 0$ finite.  Let $\mathcal O_1,\ldots, \mathcal O_k$ be the orbits of $\mathscr G$, let $G_i$ be the isotropy group of $\mathcal O_i$ (well defined up to isomorphism) and let $n_i=|\mathcal O_i|$, for $i=1,\ldots, k$.  Then \[R\mathscr G\cong \prod_{i=1}^k M_{n_i}(RG_i)\] for any commutative ring with unit $R$.  In particular, $R\mathscr G$ is Noetherian if and only if each $RG_i$ is Noetherian, for $i=1,\ldots, k$,  $R\mathscr G$ is Artinian if and only if $R$ is Artinian and  $\mathscr G$ is finite and $R\mathscr G$ is semisimple if and only if $\mathscr G$ is finite and  $R$ is a finite direct product of fields whose characteristics do not divide the order of any $G_i$.
\end{Prop}
\begin{proof}
Since $\dom$ is a local homeomorphism and $\mathscr G\skel 0$ is discrete, it follows that $\mathscr G\skel 1$ is discrete, i.e., $\mathscr G$ is a discrete groupoid with finitely many objects. Thus $R\mathscr G$ is the usual category algebra of this groupoid~\cite{ringoids,Webb} and the isomorphism is then folklore (cf.~\cite[Theorem~8.15]{repbook} where it is proved for finite groupoids but only finiteness of the set of objects is used in the proof).  Namely, one fixes a basepoint $x_i\in \mathcal O_i$ (and assumes $G_i=G_{x_i}$)  and chooses, for each $y\in \mathcal O_i$, an arrow $g_y\colon x_i\to y$.  The isomorphism sends an arrow $g\colon y\to z$ in $\mathcal O_i$ to $g_z\inv gg_yE_{zy}\in M_{n_i}(RG_i)$ where we index the rows and columns of matrices in the factor $M_{n_i}(RG_i)$ by $\mathcal O_i$.

Since a finite product of rings is Noetherian (respectively, Artinian) if and only if each factor is and a matrix algebra is Noetherian (respectively, Artinian) if and only if the base of the matrix algebra is, we deduce that $R\mathscr G$ is Noetherian (respectively, Artinian) if and only if each $RG_i$ is, for $i=1,\ldots k$. To complete the proof it suffices to apply a result of Connell that states that a group algebra $RG$ is Artinian if and only if $R$ is Artinian and $G$ is finite~\cite{connell}, to apply Maschke's theorem and to observe that a groupoid is finite if and only if it has finitely many objects and each isotropy group is finite (in which case, it has cardinality $\sum_{i=1}^k n_i^2|G_i|$ using the above notation).
\end{proof}

Let us prove a topological lemma, which is essentially a well-known result about Boolean algebras.

\begin{Lemma}\label{l:topology}
Let $X$ be a Hausdorff space with a basis of compact open sets.  Then $X$ satisfies the ascending chain condition on compact open subsets if and only if $X$ is finite.
\end{Lemma}
\begin{proof}
Clearly, if $X$ is finite, then it satisfies the ascending chain condition on compact open subsets.  Assume now that $X$ satisfies the ascending chain condition on compact open subsets.    First observe that $X$ is compact. Indeed, $X$ must have a maximal compact open subset $K$. If $K\neq X$ and $x\in X\setminus K$, then there is a compact open neighborhood $U$ of $x$ and $K\cup U\supsetneq K$ is a strictly larger compact open subset.  Thus $X=K$.  It follows from compactness of $X$ that the compact open subsets are closed under complement and hence $X$ also enjoys descending chain condition on compact open subsets.  Hence each point $x\in X$ is contained in a minimal compact open subset $K_x$. Suppose that $y\in K_x\setminus \{x\}$. Then since $X$ is Hausdorff with a basis of compact open subsets, there is a compact open subset $V$ with $x\in V\subseteq K_x$ and $y\notin V$.  This contradicts the minimality of $K_x$ and we deduce that $K_x=\{x\}$.  Thus $X$ is discrete and compact, whence finite.
\end{proof}

Recall that if $A$ is an ring and $e,f\in A$ are idempotents, then $Ae\subseteq Af$ if and only if $ef=e$.
\begin{Prop}\label{p:accbool}
Let $\mathscr G$ be an ample groupoid and suppose that $R\mathscr G$ is Noetherian for some commutative ring with unit $R$.  Then $\mathscr G\skel 0$ is finite.
\end{Prop}
\begin{proof}
Let $U,V\subseteq \mathscr G\skel 0$ be compact and open. Then $\chi_U,\chi_V\in R\mathscr G$ are idempotents and $R\mathscr G\chi_U\subseteq R\mathscr G\chi_V$ if and only if $\chi_{U\cap V}=\chi_U\ast \chi_V=\chi_U$, that is, if and only if $U\subseteq V$.  It follows that if $R\mathscr G$ is Noetherian, then $\mathscr G\skel 0$ satisfies the ascending chain condition on compact open subsets and hence is finite by Lemma~\ref{l:topology}.
\end{proof}

Theorem~\ref{t:main} is now an immediate application of Proposition~\ref{p:accbool} and Proposition~\ref{p:finitelymanyobs}.

\section{Applications}
In this section, we use Theorem~\ref{t:main} to characterize the Noetherian and Artinian properties for Leavitt path algebras and for inverse semigroup algebras.  The  results for Leavitt path algebras are due to Abrams, Aranda~Pino and Siles~Molina~\cite{LeavittNoeth,LeavittArtinian}, in the special case of coefficients in a field; for inverse semigroup algebras the Noetherian result is due to Okni\'nski~\cite{Noetherianokninskiinverse} and the Artinian result is a special case of Zelmanov's results~\cite{Zelmanovsgp}.

\subsection{Leavitt path algebras}
Let $E=(E\skel 0,E\skel 1)$ be a (directed) graph (or quiver) with vertex set $E\skel 0$ and edge set $E\skel 1$.  We use $\sour(e)$ for the source of an edge $e$ and $\ran(e)$ for the range, or target, of an edge.   A vertex $v$ is called a \emph{sink} if $\sour\inv(v)=\emptyset$ and it is called an \emph{infinite emitter} if $|\sour\inv(v)|=\infty$.  The length of a finite (directed) path $\alpha$ is denoted $|\alpha|$.

The \emph{Leavitt path algebra}~\cite{Leavitt1,LeavittPardo,Abramssurvey,Leavittbook} $L_R(E)$ of $E$ with coefficients in the unital commutative ring $R$ is the $R$-algebra generated by a set $\{v\in E\skel 0\}$ of pairwise orthogonal idempotents and a set of variables $\{e,e^*\mid e\in E\skel 1\}$ satisfying the relations:
\begin{enumerate}
\item $\sour(e)e=e=e\ran(e)$ for all $e\in E\skel 1$;
\item $\ran(e)e^*= e^*=e^*\sour(e)$ for all $e\in E\skel 1$;
\item $e^*e'=\delta_{e,e'}\ran(e)$ for all $e,e'\in E\skel 1$;
\item $v=\sum_{e\in \sour^{-1}(v)} ee^*$ whenever $v$ is not a sink and not an infinite emitter.
\end{enumerate}

It is well known that $L_R(E)=R\mathscr G_E$  for the graph groupoid $\mathscr G_E$ defined as follows.  Let $\partial E$ consist of all one-sided infinite paths in $E$ as well as all finite paths $\alpha$ ending in a vertex $v$ that is either a sink or an infinite emitter.  If $\alpha$ is a finite path in $E$ (possibly empty), put $Z(\alpha)=\{\alpha\beta\in \partial E\}$.  Note that $Z(\alpha)$ is never empty.  Then a basic open neighborhood  of $\partial E$ is of the form $Z(\alpha)\setminus (Z(\alpha e_1)\cup\cdots \cup Z(\alpha e_n))$ with $e_i\in E\skel 1$, for $i=1,\ldots n$ (and possibly $n=0$).  These neighborhoods are compact open.

  The graph groupoid $\mathscr G_E$ is the given by:
\begin{itemize}
\item $\mathscr G_E\skel 0= \partial E$;
\item $\mathscr G_E\skel 1 =\{ (\alpha\gamma,|\alpha|-|\beta|,\beta\gamma)\in \partial E\times \mathbb Z\times \partial E\}\mid |\alpha|,|\beta|<\infty\}$.
\end{itemize}
One has $\dom(\eta,k,\gamma)=\gamma$, $\ran(\eta,k,\gamma)=\eta$ and $(\eta,k,\gamma)(\gamma,m,\xi) = (\eta, k+m,\xi)$.  The inverse of $(\eta,k,\gamma)$ is $(\gamma,-k,\eta)$.

A basis of compact open subsets for the topology on $\mathscr G_E\skel 1$ can be described as follows. Let $\alpha,\beta$ be finite paths ending at the same vertex and let $U\subseteq Z(\alpha)$, $V\subseteq Z(\beta)$ be compact open with $\alpha\gamma\in U$ if and only if $\beta\gamma\in V$.  Then the set \[(U,\alpha,\beta,V)=\{\alpha\gamma,|\alpha|-|\beta|,\beta\gamma)\mid \alpha\gamma\in U,\beta\gamma\in V\}\] is a basic compact open set of $\mathscr G_E\skel 1$.
Of special importance are the compact open sets $Z(\alpha,\beta) = (Z(\alpha),\alpha,\beta,Z(\beta))= \{(\alpha\gamma,|\alpha|-|\beta|,\beta\gamma)\in \mathscr G\skel 1\}$ where $\alpha,\beta$ are finite paths ending at the same vertex.

There is an isomorphism $L_R(E)\to R\mathscr G_E$ that sends $v\in E\skel 0$ to the characteristic function of $Z(\varepsilon_v,\varepsilon_v)$ where $\varepsilon_v$ is the empty path at $v$ and, for $e\in E\skel 1$, it sends $e$ to the characteristic function of $Z(e,\varepsilon_{\ran(e)})$ and $e^*$ to the characteristic function of $Z(\varepsilon_{\ran(e)},e)$, cf.~\cite{operatorguys1,groupoidprimitive,CRS17,Strongeffective}.

In~\cite{LeavittArtinian} the finite dimensional Leavitt algebras over a field were characterized.  The Noetherian Leavitt path algebras over a field were determined in~\cite[Theorem~3.10]{LeavittNoeth}.  We extend these results to arbitrary base rings using groupoid methods (the original proofs were purely algebraic).

By a \emph{cycle} in a directed graph $E$, we mean a simple, directed,  closed circuit.  A cycle is said to have an \emph{exit} if some vertex on the cycle has out-degree at least two.  The graph $E$ is said to satisfy \emph{condition (NE)} if no cycle in $E$ has an exit.  The following is presumably well known but I do not know a reference.

\begin{Prop}\label{p:noexit}
Let $E$ be a graph.  Then $\mathscr G_E\skel 0$ is finite if and only if $E$ is finite and satisfies condition (NE).
\end{Prop}
\begin{proof}
Suppose first that $\mathscr G_E\skel 0=\partial E$ is finite.  Then $E\skel 0$ must be finite as the cylinder sets $Z(\varepsilon_v)$ with $v\in E\skel 0$ are pairwise disjoint and non-empty.  As the cylinder sets $Z(e)$ with $e\in E\skel 1$ are pairwise disjoint and non-empty, we deduce that  $E\skel 1$ is finite and hence $E$ is finite. Suppose now that some cycle contains an exit. Then there is a vertex $v$ of out-degree at least two such that there is a cycle $\alpha$ starting at $v$.  Let $e$ be an edge emitted by $v$ not belonging to $\alpha$.  Then the cylinder sets $Z(\alpha^ne)$ with $n\geq 0$ are non-empty and pairwise disjoint, again contradicting that $\partial E$ is finite.  Thus $E$ satisfies condition (NE).

Conversely, suppose that $E$ is finite and no cycle has an exit.  Then any path of length greater than $|E\skel 0|$ must enter a cycle and it can never leave that cycle.  Thus $\partial E$ consists of those paths of length at most $|E\skel 0|$ ending at a sink and those infinite paths of the form $\alpha\beta\beta\cdots$ with $\alpha$ a path of length at most $|E\skel 0|$ and $\beta$ a cycle.   As a finite graph has only finitely many cycles (as cycles do not repeat vertices),  we conclude that $\mathscr G_E\skel 0=\partial E$ is finite.
\end{proof}

The next proposition, characterizing isotropy in graph groupoids, should also be considered folklore.

\begin{Prop}\label{p:isotropy}
Let $E$ be a graph and $\gamma\in \partial E$. Then the isotropy group $G_\gamma$ is trivial unless $\gamma=\rho\zeta\zeta \cdots$ where $\rho$ is a path and $\zeta$ is a cycle, in which case $G_\gamma\cong \mathbb Z$.
\end{Prop}
\begin{proof}
An isotropy group element is of the form $g=(\gamma,k,\gamma)$ where $\gamma=\alpha\xi=\beta\xi$ with $k=|\alpha|-|\beta|$.  Moreover, $g$ is a unit unless $k\neq 0$. If $k\neq 0$, replacing $g$ by its inverse we may assume that $|\alpha|>|\beta|$.  Then $\alpha = \beta \eta$ and $\xi = \eta\xi=\eta\eta\cdots$. We thus deduce that $\gamma=\rho\zeta\zeta\cdots$ with $\rho$ a path and $\zeta$ a cycle.  Moreover, $G_\gamma = \{\gamma\}\times H\times \{\gamma\}\cong H\cong \mathbb Z$ with $H$ a non-trivial subgroup of $\mathbb Z$.
\end{proof}

As $R\mathbb Z\cong R[x,x\inv]$, the Laurent polynomial ring in one-variable over $R$, we immediately obtain from Proposition~\ref{p:noexit} and Proposition~\ref{p:isotropy}, Hilbert's basis theorem and Theorem~\ref{t:main} the following result, generalizing~\cite{LeavittArtinian} and~\cite[Theorem~3.10]{LeavittNoeth}.

\begin{Thm}
Let $E$ be a directed graph (i.e., quiver), $R$ a commutative ring with unit and $L_R(E)$ the Leavitt path algebra of $E$ over $R$.
\begin{enumerate}
\item $L_R(E)$ is Noetherian if and only if $R$ is Noetherian, $E$ is finite and no cycle in $E$ has an exit, i.e., $E$ satisfies condition (NE). Moreover, in this case $L_R(E)$ is a finite direct product of matrix algebras over $R$ and $R[x,x\inv]$.
\item $L_R(E)$ is Artinian if and only if $R$ is Artinian and $E$ is finite acyclic, in which case  it is a finite direct product of matrix algebras over $R$.
\item $L_R(E)$ is semisimple if and only if $R$ is a finite direct product of fields and $E$ is finite acyclic.
\end{enumerate}
\end{Thm}

\subsection{Inverse semigroups}
An \emph{inverse semigroup} is a semigroup $S$ such that, for each $s\in S$, there exists a unique $s^*$ in $S$ with $ss^*s=s$, $s^*ss^*=s^*$; see~\cite{Lawson} for an introduction to inverse semigroup theory.  Connections between inverse semigroups, \'etale groupoids and operator algebras can be found in~\cite{Renault,Paterson,Exel}.  In particular, each ample groupoid algebra is a quotient of an inverse semigroup algebra~\cite{mygroupoidalgebra}.  It is shown in~\cite[Theorem~6.3]{mygroupoidalgebra} that if $R$ is a commutative ring with unit and $S$ is an inverse semigroup, then $RS\cong R\mathscr G(S)$ where $\mathscr G(S)$ is Paterson's universal groupoid of $S$~\cite{Paterson}, a certain ample groupoid associated to $S$~\cite{Paterson,Exel,mygroupoidalgebra}.   The full description of this groupoid is a bit complicated but we describe some of its salient features.

 The set $E(S)$ of idempotents of $S$ is a commutative subsemigroup~\cite{Lawson}. The unit space of $\mathscr G(S)$ is the set of homomorphisms from $E(S)$ to  $\{0,1\}$ (with the latter viewed as a semigroup under multiplication).  Thus $\mathscr G(S)\skel 0$ is finite if and only if $E(S)$ is finite as the homomorphisms from an idempotent, commutative semigroup (i.e., a meet semilattice) to $\{0,1\}$ separate points. Moreover, when $E(S)$ is finite  $\mathscr G(S)\skel 1$ is in bijection with $S$, the isotropy groups are precisely the maximal subgroups of $S$ and, in fact, $\mathscr G(S)$ reduces to the so-called underlying groupoid of the inverse semigroup considered in~\cite{Lawson} or~\cite{mobius2}; see~\cite{mygroupoidalgebra} for details.  Recall that a \emph{maximal subgroup} of $S$ is a unit group of a monoid $eSe$ with $e\in E(S)$. Thus Theorem~\ref{t:main} recovers Okni\'nski's theorem~\cite{Noetherianokninskiinverse} (see also~\cite[Ch.~12, Cor.~9]{oknisemigroupalgebra}) and Zelmanov's theorem~\cite{Zelmanovsgp} (restricted to inverse semigroups).

\begin{Thm}
Let $S$ be an inverse semigroup and $R$ a commutative ring with unit.
\begin{enumerate}
  \item $RS$ is Noetherian if and only if $S$ has finitely many idempotents and  $RG$ is Noetherian for each maximal subgroup $G$ of $S$.
  \item $RS$ is Artinian if and only if $R$ is Artinian and $S$ is finite.
  \item $RS$ is semisimple if and only if $S$ is finite and $R$ is a finite direct product of fields whose characteristics  divide the order of no maximal subgroup of $S$.
\end{enumerate}
In any of these cases, $RS$ is isomorphic to a finite direct product of matrix algebras over group algebras of maximal subgroups.
\end{Thm}

\def\cprime{$'$} \def\cprime{$'$} \def\cprime{$'$} \def\cprime{$'$}
  \def\cprime{$'$} \def\cprime{$'$} \def\cprime{$'$} \def\cprime{$'$}
  \def\cprime{$'$} \def\cprime{$'$}

\end{document}